\documentclass[12pt,reqno]{amsproc}

\usepackage{amsmath}
\usepackage{amssymb,cite, enumerate}
\usepackage{wasysym}

\renewcommand{\a}{\alpha} 
 
\numberwithin{equation}{section}
\theoremstyle{plain}

\newtheorem{theorem}{Theorem}[section]

\newtheorem{remark}[theorem]{Remark}

\theoremstyle{definition}
\newtheorem{definition}[theorem]{Definition}

\allowdisplaybreaks

\begin{document}

\title{Toeplitz determinants associated with  Ma-Minda classes of starlike and convex functions}

\author{Om P. Ahuja}
\address{Department of Mathematics, Kent State University,
Ohio, USA }
\email{oahuja@kent.edu},

\author{Kanika Khatter}
\address{Department of Mathematics, Hindu Girls College,
Sonipat, Haryana, India}
\email{kanika.khatter@yahoo.com}

\author{V. Ravichandran}
\address{Department of Mathematics, National Institute of Technology, Tiruchirappalli 620 015,
India}
\email{vravi68@gmail.com}

\maketitle

\begin{abstract}
A starlike function $f$ is  characterized by the quantity $zf'(z)/f(z)$ lying in  the right half-plane. This paper deals with sharp bounds for certain   Toeplitz determinants whose entries are the coefficients of the functions  $f$  for which the quantity $zf'(z)/f(z)$ takes values in certain specific subset in the right half-plane.   The results obtained include several new special cases and some known results.
\end{abstract}

\section{Introduction}
Let $\mathbb{D}$ be the open unit disk in $\mathbb{C}$ and let $\mathcal{A}$ be the class of all analytic functions $f:\mathbb{D}\to\mathbb{C}$ having Taylor series $  f(z)= z+ \sum_{n=2}^{\infty} a_n z^n$.
Let  $\mathcal{S}$ be the well known subclass of $\mathcal{A}$ of univalent ($\equiv$ one-to-one) functions. A set $D$ is starlike with respect to $0\in D$ if $tw\in D$ for all  $w\in D$ and for all $t$ with $0\leq t \leq  1$; it is  convex if $tw_1+(1-t)w_2\in D$ for all $w_1,w_2\in D$ and for all $t$ with $0\leq t\leq 1$.  The subclasses of $\mathcal{S}$ consisting of functions
$f$  for which $f(\mathbb{D})$  is starlike with respect to the origin    and  convex   are denoted respectively by $\mathcal{S}^*$ and  $\mathcal{K}$.
These classes were introduced and studied aiming at a proof of the famous coefficient conjecture of  Bieberbach that $|a_n|\leq n$ with equality for the Koebe function $z/(1-z)^2$ or its rotations; see the  survey article by Ahuja \cite{AA} and several references therein for a history on the problem. The concept of subordination is useful in unifying various subclasses of univalent functions. First, let us denote by $\Omega$ the class of all analytic functions $w:\mathbb{D}\to \mathbb{D}$ with $w(0)=0$. A function in $\Omega $ is known as a Schwarz function.  An analytic function $f$ is said to be subordinate to the analytic function $F$,  written $f\prec F$ or $f(z)\prec F (z),\quad (z\in\mathbb{D})$ if there exists a   function $w\in \Omega$   such that $f(z)=F(w(z))$ for all  $z\in \mathbb{D}$. If the function $F $ is univalent in $\mathbb{D}$,  then the subordination $f(z)\prec F (z)$ holds if and only if $f(0)=F(0)$ and $f(\mathbb{D})\subseteq F(\mathbb{D})$. The class  $\mathcal{P}$ of Caratheodory functions consists of all  analytic functions $p:\mathbb{D}\to \mathbb{C}$ with $\operatorname{Re} p(z)>0$ for  $z\in \mathbb{D}$. The two classes are closely associated as a  function $p\in \mathcal{P}$ if and only if there is a $w\in \Omega$ with $p=(1+w)/(1-w)$. These functions are characterized analytically as follows:
\begin{align*}
  \mathcal{S}^* &= \Big\{ f\in \mathcal{A}: \operatorname{Re}\Big( \frac{z f'(z)}{f(z)}\Big)>0 \Big\} \\
  &=\Big\{ f\in \mathcal{A}:  \frac{z f'(z)}{f(z)}\prec \frac{1+z}{1-z} \Big\},    \intertext{and}
  \mathcal{K}  &= \Big\{ f\in \mathcal{A}: \operatorname{Re}\Big(1 + \frac{z f''(z)}{f'(z)}\Big)>0 \Big\}\\
   &=   \Big\{ f\in \mathcal{A}:  1 + \frac{z f''(z)}{f'(z)}\prec \frac{1+z}{1-z} \Big\}.
\end{align*}

Ma and Minda \cite{mam} gave a  unified treatment of distortion, growth and covering theorems for the  functions $f\in \mathcal{S}^*$ and $f\in \mathcal{K}$ for which either of the quantity $zf'(z)/f(z)$ or $1+ z f''(z)/f'(z)$ is subordinate to a more general subordinate function $\varphi\in\mathcal{P}$. In \cite{mam}, it is assumed that the function $\varphi$ is starlike and the image of unit disk is symmetric with respect to real axis. However, we do not require these conditions in this paper.

\begin{definition} For an  analytic univalent function    $\varphi$  with positive real part in $\mathbb{D}$,  $\varphi(0)= 1$, $\varphi'(0) >0$ and $\varphi''(0)\in \mathbb{R}$, the classes $\mathcal{S}^*(\varphi)$ and $ \mathcal{K}(\varphi)$ are defined by
\begin{align*}
 \mathcal{S}^*(\varphi)&:=\left\{ f \in \mathcal{S}: \frac{zf'(z)}{f(z)}\prec \varphi(z) \right\}\intertext{and}
 \mathcal{K}(\varphi)&:=  \left\{ f\in \mathcal{S}: 1+\frac{zf''(z)}{f'(z)}\prec \varphi(z) \right\}.
\end{align*}
\end{definition}

Toeplitz matrices and their determinants play an important role in several branches of mathematics and have many applications \cite{BB}. For information on  applications of Toeplitz matrices to several areas of pure and applied mathematics, we refer to the survey article by Ye and Lim \cite{lim}. We recall that Toeplitz symmetric matrices have constant entries along the diagonal.  For the function  $  f(z)= z+ \sum_{n=2}^{\infty} a_n z^n$, we associate a   determinant  $T_q(n)$ defined by
\begin{equation*}
T_q(n) := \left| \begin{array}{cccc}
a_n      & a_{n+1} &  \cdots & a_{n+q-1} \\
a_{n+1}  & a_{n} &  \cdots & a_{n+q}\\
\vdots   & \vdots  &         & \vdots  \\
a_{n+q-1}& a_{n+q} &  \cdots & a_{n}
 \end{array} \right|.
\end{equation*}
In 2017, Ali \emph{et al.}\ \cite{halim}   studied   Toeplitz  determinants $ T_q(n)$ for initial values of $n$ and $q$, where the entries of $T_q(n)$ are the coefficients of the functions that are starlike, convex and close to convex. Motivated by Ali \emph{et al.}\  \cite{halim}, some researchers in the last three years studied $T_q(n)$ for low values of $n$ and $q$, where entries are the coefficients of functions in several subclasses of analytic functions.  Some recent work on coefficient problems includes \cite{cho2,cudna,lec1,lec2}.

In this paper, we obtain sharp estimates for Toeplitz determinants $ T_2(2) $ and $ T_3(1)$ for functions belonging to the classes $\mathcal{S}^*(\varphi)$ and $\mathcal{K}(\varphi)$. The functions $K_{\varphi}$ and $H_{\varphi}$ defined by
\begin{equation}\label{K1}
  \frac{z K_{\varphi}'(z)}{K_{\varphi}(z)} = \varphi(i z), \quad   K_{\varphi}(0)= K_{\varphi}'(0)-1=0
\end{equation}
and
\begin{equation}\label{H1}
 1+  \frac{z H_{\varphi }''(z)}{H_{\varphi }'(z)} = \varphi( i z), \quad H_{\varphi }(0)= H_{\varphi}'(0)-1=0
\end{equation}
respectively belong to the classes $\mathcal{S}^*(\varphi)$ and $\mathcal{K}(\varphi)$. We shall  use these functions to demonstrate sharpness in certain cases.   For a function $p\in \mathcal{P}$ with  $ p(z)=  1+ c_1 z + c_2 z^2 + c_3 z^3 +\cdots$,  it is well-known \cite{grenander} that $|c_n|  \leq 2$. The main results are proved by using this estimate by  associating coefficients of the functions in our classes to the functions in the class $\mathcal{P}$. We shall also use estimates for the Fekete-Szeg\"o functional for the two classes $\mathcal{S}^*(\varphi)$ and $\mathcal{K}(\varphi)$ from Ali et al. \cite{seeni} and Ma and Minda \cite{mam}. The symmetry of the image of $\varphi$ was  used in \cite{mam} to ensure that the coefficients of $\varphi$ are real and we have assumed it here for the first two coefficients. In  \cite{mam}, the univalence was used in defining the function $p_1$ by
\[p_1(z)=\frac{1-\varphi^{-1}(p(z))}{1+\varphi^{-1}(p(z))}. \] However, this requirement can be dropped by defining $p_1$ by \eqref{1k}.

\section{Main Results}

Theorem~\ref{th1} and Theorem~\ref{th2} respectively give the sharp bound for $ T_2(2) =a_3^2 - a_2^2$   for  functions $f \in \mathcal{S}^*(\varphi)$ and $f\in  \mathcal{K}(\varphi)$.
\begin{theorem}\label{th1}
  If $f \in \mathcal{S}^*(\varphi)$ and $\varphi(z)= 1+ B_1 z + B_2 z^2 + \cdots$ with  $0<B_1\leq |B_2+B_1^2|$, then the Toeplitz determinant $T_2(2)$ satisfies the sharp bound:
  \begin{equation*}
    |T_2(2)| \leq \frac{1}{4}(B_2 + B_1^2)^2 + B_1^2.
  \end{equation*}
\end{theorem}

\begin{proof}
  Since $f \in \mathcal{S}^*(\varphi)$,  there is a function $w$ in the class $\Omega$ of Schwarz functions   satisfying
  that
  \begin{equation}\label{4k}
   \frac{z f'(z)}{f(z)}= \varphi(w(z)).
  \end{equation}
  Corresponding to the function $w$,  define the function $p_1:\mathbb{D}\to\mathbb{C}$   by
 \begin{equation}\label{1k}  p_1(z)=\dfrac{1+ w(z)}{1-w(z)}=1 + c_1 z + c_2 z^2 + \cdots  \end{equation}
so that
\begin{equation}\label{1ka}  w(z)= \dfrac{p_1(z)-1}{p_1(z)+1}= \dfrac{1}{2} c_1 z +\dfrac{1}{2}\left(c_2-\dfrac{1}{2}c_1^2\right)z^2+\cdots.
 \end{equation}
 Clearly, the function $p_1$ is analytic in $\mathbb{D}$ with $p_1(0)=1$. Since $w\in \Omega$, it follows that  $ p_1\in \mathcal{P}$. Using \eqref{1ka} and the Taylor series of $\varphi$ given by
 $\varphi(z)= 1 + B_1 z + B_2 z^2 + B_3 z^3+ \cdots$,  we get
 \begin{equation}\label{2K*}
 \varphi\left(w(z)\right)=1 + \dfrac{1}{2}B_1 c_1z + \left(\dfrac{1}{2}B_1\left(c_2-\dfrac{1}{2}c_1^2\right)+\dfrac{1}{4}B_2 c_1^2\right)z^2 +\cdots.
 \end{equation}
Since  $f(z)= z+ a_2 z^2 + a_3 z^3 + \cdots$,  the Taylor series expansion of  the function $zf'/f$ is given by
 \begin{align}\label{6k}
  \frac{z f'(z)}{f(z)}&=1 + a_2 z + (-a_2^2 + 2 a_3) z^2 + (a_2^3 - 3 a_2 a_3 +
    3 a_4) z^3 \\& \quad + (-a_2^4 + 4 a_2^2 a_3 - 2 a_3^2 - 4 a_2 a_4 + 4 a_5) z^4 + \cdots.\notag
 \end{align}
 Using \eqref{4k}, \eqref{2K*} and \eqref{6k}, the coefficients $a_2$ and $a_3$ can be expressed as a function of the coefficients $c_i$ of $p \in \mathcal{P}$ and $B_i$ of $\varphi$  as follows:
\begin{align}
  a_2 &= \frac{1}{2} B_1 c_1 \label{a2} \intertext{and}
  a_3 &= \frac{1}{8}\big(( B_1^2 -B_1  + B_2 )c_1^2 + 2 B_1 c_2\big)\label{a3}.
\end{align}
The equations  \eqref{a2} and \eqref{a3} (see Ali et al. \cite{seeni} for a general result for $p$-valent functions) readily shows that
\begin{align}\label{fs1}
|a_3 - \mu a_2^2|\leq \begin{cases}
                          \ &\dfrac{1}{2}\big(B_2 +  B_1^2 -2 \mu B_1^2\big),  \quad  \text{ if } 2 B_1^2 \mu \leq B_2 + B_1^2 -B_1; \\[14pt]
                          \ &\dfrac{1}{2} B_1, \quad  \text{ if }    B_2+ B_1^2 - B_1   \leq 2 B_1^2 \mu \leq B_2 + B_1^2 +B_1; \\[14pt]
                         \ & \dfrac{1}{2}\big(-B_2 -  B_1^2 + 2 \mu B_1^2\big),\quad   \text{ if } B_2 + B_1^2 +B_1 \leq 2 B_1^2  \mu.
                        \end{cases}
\end{align}
Since  $|c_n| \leq 2$, the equation \eqref{a2} shows that
\begin{align} \label{ba2}
|a_2| \leq  B_1
\end{align}  and, when  $B_1\leq |B_2+B_1^2|$, the equation \eqref{fs1} readily yields
\begin{align}\label{ba3}
|a_3|  \leq \frac{1}{2}\big| B_1^2 + B_2 \big|
\end{align}
Using these estimates for the second and third coefficients given in \eqref{ba2} and \eqref{ba3}, we have
\begin{equation*}
 |a_3^2 - a_2^2| \leq |a_3|^2 + |a_2|^2 \leq \frac{1}{4}\big( B_1^2 + B_2\big)^2 + B_1^2.
\end{equation*}

The result is sharp for the function $K_\varphi$ given by \eqref{K1}. This function $K_\varphi$ has the   Taylor series given by
\[K_\varphi(z) =z-iB_1z^2-\frac{1}{2}(B_1^2+B_2)z^3+\dotsb.\]
The Taylor series can be obtained by noting that $K_\varphi$ corresponds to the function $f$ given by \eqref{4k} when $w(z)=iz$. In this case, $p_1(z)=1+2iz-2z^2+\dotsb$. With $c_1=2i$ and $c_2=-2$, we get $a_2=iB_1$ and $a_3=- (B_1^2+B_2)/2$. Clearly, for the function $K_\varphi$, we have
\[ |a_3^2 - a_2^2| = \frac{1}{4}\big( B_1^2 + B_2\big)^2 + B_1^2\]
proving the sharpness.
\end{proof}

\begin{theorem}\label{th2}
  If $f \in \mathcal{K}(\varphi)$ and $\varphi(z)= 1+ B_1 z + B_2 z^2 + \cdots$ with $0<B_1\leq |B_2+B_1^2|$, then the Toeplitz determinant $T_2(2)$ satisfies the sharp bound given by
  \begin{equation*}
    |T_2(2)| \leq \frac{1}{36}\big( B_1^2 + B_2\big)^2 + \frac{1}{4} B_1^2.
    \end{equation*}
\end{theorem}

\begin{proof}
Let $f(z)= z +a_2 z^2 +a_3 z^3 +\cdots$ and $\varphi(z)= 1+ B_1 z + B_2 z^2 + \cdots$.
Since $f \in \mathcal{K}(\varphi)$, there is a function $w$ in the class $\Omega$ of Schwarz functions   such
  that
  \begin{equation}\label{4ka}
  1+  \frac{z f''(z)}{f'(z)}= \varphi(w(z)).
  \end{equation}
  The Taylor series expansion of the function $f$ given by $f(z)= z+ a_2 z^2 + a_3 z^3 + \cdots$ shows that
 \begin{align}\label{6K'}
 1+\frac{z f''(z)}{f'(z)}&= 1 + 2 a_2 z + (-4 a_2^2 + 6 a_3) z^2 +\cdots.
 \end{align}
 Then using \eqref{4ka}, \eqref{6K'} and \eqref{2K*}, the coefficients $a_2$ and  $a_3$ can be expressed as a function of the coefficients $c_i$ of $p \in \mathcal{P}$ given by
\begin{align*}
  a_2 &= \frac{1}{4}  B_1 c_1,\intertext{and}
  a_3 &= \frac{1}{24}\big((-B_1 + B_1^2 + B_2) c_1^2 + 2 B_1 c_2\big).
\end{align*}
Using the well-known estimate $|c_n| \leq 2$ for the function $p_1$ with positive real part,  it follows that
\begin{align}\label{A2'}
|a_2| &\leq  \frac{B_1}{2}.
\end{align}
 For a function   $f\in \mathcal{K}(\varphi)$, Ma and Minda \cite{mam} proved     the following inequality
\begin{align}\label{fs2}
|a_3 - \mu a_2^2|\leq \left\{
                        \begin{array}{ll}
                          \dfrac{1}{6}\big(B_2 - \dfrac{3}{2}\mu B_1^2 + B_1^2\big), \quad  \text{if\ }  3 B_1^2 \mu \leq 2(B_2 + B_1^2 -B_1) ; \\ \ \\
                          \dfrac{1}{6} B_1,   \quad \text{if\ }  2 (B_2+ B_1^2 - B_1) \leq 3 B_1^2 \mu \leq 2(B_2 + B_1^2 +B_1) ;  \\ \ \\
                           \dfrac{1}{6}\big(-B_2 + \dfrac{3}{2}\mu B_1^2 - B_1^2\big), \quad  \text{if\ }   2(B_2 + B_1^2 +B_1) \leq 3 B_1^2  \mu .
                        \end{array}
                      \right.
\end{align}
 Since $B_1\leq |B_2+B_1^2|$, the inequality \eqref{fs2} readily gives
\begin{align}
|a_3| &\leq \frac{1}{6}\big|  B_1^2 +  B_2 \big|\label{A3'}.
\end{align}
Using the bound for $a_2$ and $a_3$ given respectively by \eqref{A2'} and \eqref{A3'}, we get
\begin{equation*}
 |a_3^2 - a_2^2| \leq |a_3|^2 + |a_2|^2 \leq \frac{1}{36}\big( B_1^2 + B_2\big)^2 + \frac{1}{4} B_1^2.
\end{equation*}

The result is sharp for the function $H_{\varphi}$ defined in \eqref{H1}. Indeed, for this function $H_\varphi$, we have $a_2=B_1i/2$ and $a_3=-(B_1^2+B_2)/6$ and hence
\begin{equation*}
 |a_3^2 - a_2^2| = \frac{1}{36}\big( B_1^2 + B_2\big)^2 + \frac{1}{4} B_1^2
\end{equation*}
proving the sharpness of the result.
\end{proof}

Theorem~\ref{th3} and Theorem~\ref{th4} give  the sharp bound for the  Toeplitz determinant $T_3(1)$  for  functions respectively  in the classes $\mathcal{S}^*(\varphi)$ and  $\mathcal{K}(\varphi)$.

\begin{theorem}\label{th3}
   If $f \in \mathcal{S}^*(\varphi)$ and $\varphi(z)= 1+ B_1 z + B_2 z^2 + \cdots$, with $B_1>0$ and     $ B_1 -B_1^2\leq B_2\leq  3B_1^2-B_1$, then the Toeplitz determinant $T_3(1)$ satisfies the sharp bound:
\[ |T_3(1)|  \leq  1 + 2 B_1^2 + \dfrac{1}{4}(B_2 + B_1^2)(3 B_1^2- B_2).\]
\end{theorem}

\begin{proof}Since
\[T_3(1) = \left|\begin{array}{ccc}
1 & a_2 & a_3 \\
 a_2 & 1 & a_2 \\
 a_3 & a_2 & 1
\end{array}\right| =  1- 2 a_2^2-a_3 (a_3- 2a_2^2)  \]
  it follows  that
\begin{align}\label{T3a}
 |T_3(1)| & \leq  1+ 2 |a_2|^2 + |a_3||a_3- 2 a_2^2|.
\end{align}
 Since $B_1\leq B_1^2 +B_2$, the inequality \eqref{ba3} gives
\begin{align}\label{ba3a}
|a_3|  \leq \frac{1}{2}\big( B_1^2 + B_2 \big).
\end{align}
Since  $B_1+B_2 \leq  3B_1^2$, the equation \eqref{fs1} readily yields
\begin{align}\label{ba4a}
|a_3-2a_2^2|  \leq \frac{1}{2}\big(3 B_1^2- B_2 \big).
\end{align}
Using these estimates for the second and third coefficients given in \eqref{ba2} and \eqref{ba3a},  and  the bound for $a_3- 2 a_2^2$  given by \eqref{ba4a} in \eqref{T3a}, we obtain
\[ |T_3(1)|  \leq  1 + 2 B_1^2 + \dfrac{1}{4}(B_2 + B_1^2)( 3 B_1^2 -B_2).\]

The result is sharp for the function $K_\varphi$ given by \eqref{K1}. For this function $K_\varphi$, we have  $a_2=iB_1$ and $a_3=-(B_1^2+B_2)/2$ and
\[ 1- 2 a_2^2-a_3 (a_3- 2a_2^2)= 1 + 2 B_1^2 + \dfrac{1}{4}(B_2 + B_1^2)(3 B_1^2- B_2),\]
 proving  the sharpness of our result.
\end{proof}


\begin{theorem}\label{th4}
  If $f \in \mathcal{K}(\varphi)$ and $\varphi(z)= 1+ B_1 z + B_2 z^2 + \cdots$ with $B_1>0$ and     $ B_1 -B_1^2\leq B_2\leq  2B_1^2-B_1$, then
  the Toeplitz determinant $T_3(1)$ satisfies the sharp bound:
\begin{align*}
|T_3(1)|& \leq 1 + \dfrac{1}{2} B_1^2 + \dfrac{1}{36}( B_1^2 + B_2)(2 B_1^2-B_2 ).
\end{align*}
\end{theorem}

\begin{proof}
The given conditions on $B_1$ and $B_2$ is the same as $B_1\leq B_1^2 + B_2$ and, $B_1+B_2\leq 2 B_1^2$.
Since $B_1\leq  B_1^2 +B_2 $, the inequality \eqref{fs2} gives
\begin{align}\label{ba3b}
|a_3|  \leq \frac{1}{6}\big( B_1^2 + B_2 \big).
\end{align}
Since $B_1 \leq 2B_1^2-B_2 $, the inequality \eqref{fs2} gives
\begin{align}\label{ba3c}
|a_3-2a_2^2|  \leq \frac{1}{6}\big( 2B_1^2 - B_2 \big).
\end{align}
Using the bound for $a_2$ and $a_3$ given by \eqref{A2'} and \eqref{ba3b} and the bound for $a_3-2a_2$ given by \eqref{ba3c} in \eqref{T3a}, we get the desired result.

The result is sharp for the function $H_{\varphi}$ defined in \eqref{H1}. Indeed, for this function $H_\varphi$, we have $a_2=B_1i/2$ and $a_3=-(B_1^2+B_2)/6$ and hence
\begin{equation*}
 1- 2 a_2^2-a_3 (a_3- 2a_2^2) =  1 + \dfrac{1}{2} B_1^2 + \dfrac{1}{36}( B_1^2+B_2)(2 B_1^2-B_2 )
\end{equation*}
proving the sharpness of the result.
\end{proof}

\begin{remark} The problem of finding the sharp bound for $T_2(2)$  for functions in the classes $\mathcal{S}^*(\varphi)$ and $\mathcal{K}(\varphi)$ is open when $|B_2+B_1^2| \leq B_1$. Similarly, the determination of sharp bounds for $T_3(1)$ in other cases are open. It may be interesting to extend the results for other classes, in particular, the classes considered in \cite{aouf1} and \cite{aouf2}.
\end{remark}

\section{Some Special Cases}

Ma and Minda classes of starlike and convex functions include several well-known classes as special cases which have been studied by several authors (see for example \cite{kar, mah}). For some of these subclasses,  Theorems \ref{th1}--\ref{th4} give  the sharp bounds for  $|T_2(2)|$ and $|T_3(1)|$.

\begin{description}
\item[3.1] For $-1 \leq B <A \leq 1$, $\mathcal{S}^*[A,B]:= \mathcal{S}^*((1+ A z)/(1+ B z))$ is the familiar class consisting of Janowski starlike functions and $\mathcal{K}[A,B]:= \mathcal{K}((1+ A z)/(1+ B z))$ is the class of Janowski convex functions. These classes were initially introduced and studied by Janowski \cite{Janowski}. The series expansion of $\varphi(z) = (1+ A z)/(1 + B z)$ yields
\begin{equation*}
  \varphi(z):= \frac{1+ A z}{1+ B z}= 1 + (A - B) z + B ( B - A ) z^2 + B^2 (A - B) z^3 + \cdots
\end{equation*}
which implies $ B_1 = (A - B)$ and $B_2 = -B ( A - B )$.

If $|A-2B|\geq 1$, then, for $f\in \mathcal{S}^*[A,B]$,
\begin{align*}
|T_2(2)|& \leq (A - B)^2 (4 + A^2 - 4 A B + 4 B^2)/4
\intertext{ and for  $f\in \mathcal{K}[A,B]$}
|T_2(2)| & \leq (A - B)^2 (9 + A^2 - 4 A B + 4 B^2) /36 .
\end{align*}
If $B\leq \min\{(A-1)/2, (3A-1)/2\}$, then, for $f\in \mathcal{S}^*[A,B]$, we have
\begin{align*}
 |T_3(1)| & \leq    1 + 2 (A - B)^2 +  (3 A^2 - 5 A B + 2 B^2) (A^2 - 3 A B + 2 B^2)/4. \intertext{ If $A+B\geq 0$ and $B\leq (A-1)/2$, then, for  $f\in \mathcal{K}[A,B]$,}
  |T_3(1)|& \leq  1 + (A - B)^2/2 + (2 A^2 - 3 A B + B^2) (A^2 - 3 A B + 2 B^2)/36.
\end{align*}

The classes  $\mathcal{S}^*(\a):=\mathcal{S}^*[1-2\a,-1]$   and $\mathcal{K}(\alpha):=\mathcal{K}[1-2\a,-1]$, respectively, consisting of the starlike functions of order $\a$ and convex functions of order $\a$   were introduced and studied by Robertson \cite{rob}. For $f \in \mathcal{S}^*(\a)$, we have
\begin{align*}
|T_2(2)| & \leq (1 - \alpha)^2 (13 - 12 \alpha + 4 \alpha^2),
\intertext{and}
|T_3(1)|& \leq 24 - 74 \alpha + 91 \alpha^2 - 52 \alpha^3 + 12 \alpha^4,\quad   \hbox{$\a \leq 2/3$.}
\intertext{ For  $f \in \mathcal{K}(\a)$, we have }
|T_2(2)| &  \leq 2(1 - \alpha)^2 (9 - 6 \alpha + 2 \alpha^2)/9.
\intertext{ and }
|T_3(1)| & \leq (36 - 72 \alpha + 71 \alpha^2 - 34 \alpha^3 + 8 \alpha^4)/9 , \quad   \hbox{$\a \leq 1/2$.}
 \end{align*}
 In particular, for $f\in \mathcal{S}^*:=\mathcal{S}^*(0)$, we have $|T_2(2)|\leq 13$ and $|T_3(1)|\leq 24$. For  $f\in \mathcal{K}:=\mathcal{K}(0)$, we have $|T_2(2)|\leq 2$. Also, $|T_3(1)|\leq 4$ for $f\in \mathcal{K}$.   These bounds for starlike and convex functions were recently obtained  in \cite{halim}.


\item[3.2]  Mendiratta \emph{et al.}\cite{sumit} introduced and studied the class $\mathcal{S}^*_e = \mathcal{S}^*(e^z)$.  More generally, Khatter \emph{et al.} \cite{Khatter} defined and studied the classes $\mathcal{S}^*_{\a,e}:=\mathcal{S}^*(\a + (1-\a)e^z)$ and  $\mathcal{K}_{\a,e}:=\mathcal{K}(\a + (1-\a)e^z)$ where $0 \leq \a <1$. When $\a=0$, this classes reduce to the classes $\mathcal{S}^*_e$ and $\mathcal{K}_e$ respectively. The Taylor series   of $\varphi$ given by
  \begin{align*}
  \varphi(z) &:= \a + (1-\a)e^z =1 + (1 - \alpha) z + \frac{1}{2}(1 - \alpha) z^2 + \frac{1}{6}(1 - \alpha) z^3 + \cdots
\end{align*}
shows that  $ B_1= (1 - \alpha) $ and $ B_2 = (1 - \alpha)/2 $. For $0\leq \alpha \leq 1/2$, we get
\begin{align*}|T_2(2)| &  \leq (1 - \alpha)^2 (25 - 12 \alpha + 4 \alpha^2)/16
\intertext{and}
 |T_3(1)| &\leq (7 - 5 \alpha + 2 \alpha^2) (9 - 11 \alpha + 6 \alpha^2)/16
\end{align*}
 for $f \in \mathcal{S}^*_{\a,e}$. For $0\leq \alpha \leq 1/2$,
 we get
 \begin{align*}|T_2(2)| &  \leq (1 - \alpha)^2 (45 - 12 \alpha + 4 \alpha^2)/144
 \intertext{and}
 |T_3(1)| & \leq  (225 - 180 \alpha + 125 \alpha^2 - 34 \alpha^3 + 8 \alpha^4)/144
 \end{align*}
for $f \in \mathcal{K}_{\a,e}$. In particular,  for $f \in \mathcal{S}^*_e$, we get $|T_2(2)| \leq 25/16 \approx 1.5625$ and $|T_3(1)| \leq 63/16 \approx 3.9375$. For $f \in \mathcal{K}_{e}$, we have $|T_2(2)| \leq  5/16 \approx  0.3125$  and $|T_3(1)| \leq 25/16 \approx  1.5625$.

%
%
%

\item[3.3]  Sharma \emph{et al.}\cite{sharma} defined and studied the class of functions defined by $\mathcal{S}^*_C = \mathcal{S}^* (\varphi_c(z))$, where $\varphi_c(z) = 1 + (4/3) z + (2/3) z^2$. The geometrical interpretation is that a function $f$ belongs to the class $\mathcal{S}^*_C$ if  $z f'(z)/f(z)$ lies in the region $\Omega_c$ bounded by the cardioid i.e. $ \varphi_{c}(\mathbb{D}):= \{ x + i y : (9 x^2 + 9 y^2 -18 x + 5)^2 -16 (9x^2 + 9 y^2 -6 x +1) = 0 \}$. The convex analogous class of the above mentioned class is $\mathcal{K}_C := \mathcal{K}(\varphi_c(z))$. Its geometrical interpretation is that a function $f$ belongs to the class $\mathcal{K}_C$ if  $1+ z f''(z)/f'(z)$ lies in the region $\Omega_c$ bounded by the cardioid i.e. $ \varphi_{c}(\mathbb{D}):= \{ x + i y : (9 x^2 + 9 y^2 -18 x + 5)^2 -16 (9x^2 + 9 y^2 -6 x +1) = 0 \}$.
    When $f \in \mathcal{S}^*_{C}$, it follows that $zf'(z)/f(z) \prec  1 + (4/3) z + (2/3) z^2$ which yields $B_1= 4/3$ and $B_2 = 2/3$. And therefore $|T_2(2)| \leq 265/81  \approx 3.2716$ and $|T_3(1)| \leq 200/27  \approx 7.40741 $ for $f \in \mathcal{S}^*_{C}$.
    Whereas, $|T_2(2)| \leq 445/729  \approx 0.610425$ and $|T_3(1)| \leq 1520/729 \approx  2.08505$ for $f \in \mathcal{K}_{C}$.

\item[3.4]   Cho \emph{et al.}\cite{cho} defined and studied the class $\mathcal{S}^*_{\sin} = \mathcal{S}^*(1 + \sin z)$. The convex analogous subclass is defined as $\mathcal{K}_{\sin} := \mathcal{K}(1 + \sin z)$. Let the function $f \in \mathcal{S}^*_{\sin}$. Writing the Taylor series expansion for $\sin{z}$, we get
\begin{equation*}
  \varphi(z):= 1+  \sin{z}= 1 + z - \frac{1}{6}z^3 + \frac{1}{120} z^5 + \cdots.
\end{equation*}
Thus, $B_1= 1$ and $B_2 = 0$ which implies $|T_2(2)| \leq 5/4 = 1.25$, proved in \cite{zang}, and $|T_3(1)| \leq 15/4  =  3.75$ for  $f \in \mathcal{S}^*_{\sin}$. Similarly we can obatin $|T_2(2)| \leq 5/18  \approx  0.277778 $ and $|T_3(1)| \leq 14/9  \approx 1.55556$ for  $f \in \mathcal{K}_{\sin}$.

\item[3.5]   Raina and Sokol \cite{raina} defined the class $\mathcal{S}^*_{\leftmoon} = \mathcal{S}^*(\varphi_{\leftmoon})$, where $\varphi_{\leftmoon} = z + \sqrt{1+ z^2}$. Its convex subclass is $\mathcal{K}_{\leftmoon} := \mathcal{K}(\varphi_{\leftmoon})$. The classes $\mathcal{S}^*_{\leftmoon}$ and $\mathcal{K}_{\leftmoon}$ consist of functions for which $z f'(z)/f(z)$ and $1+ z f''(z)/f'(z)$ lies in the the leftmoon region $\Omega_{\leftmoon}$ defined by $ \varphi_{\leftmoon}(\mathbb{D}) :=\{ w\in \mathbb{C}: |w^2 - 1| < 2 |w| \}$.
   Thus, $f \in \mathcal{S}^*_{\leftmoon}$ implies $zf'(z)/f(z) \prec  z + \sqrt{1+ z^2}$ and therefore we have
\begin{equation*}
  \varphi_{\leftmoon}(z):= z + \sqrt{1+ z^2} = 1 + z + \frac{1}{2}z^2 - \frac{1}{8}z^4+ \cdots.
\end{equation*}
Therefore,  $B_1= 1$ and $B_2 = 1/2$ which immediately yields $|T_2(2)| \leq 25/16  =  1.5625$ and $|T_3(1)|\leq 63/16 = 3.9375 $ for $f \in \mathcal{S}^*_{\leftmoon}$. Similarly, $f \in \mathcal{K}_{\leftmoon}$ implies $1+ z f''(z)/f'(z) \prec z + \sqrt{1+ z^2}$, and therefore we have, $|T_2(2)| \leq 5/16 = 0.3125 $ and $|T_3(1)| \leq 25/16= 1.5625$.

%

\item[3.6]  Ronning \cite{ronning}, motivated by Goodman \cite{good}, introduced and studied the parabolic starlike class $\mathcal{S}_{P}$ and the uniformly convex class $\mathcal{UCV}$ obtained from Ma-Minda class of starlike and convex functions, respectively, by replacing
 \[\varphi(z) := 1+ \frac{2}{\pi^2} \Big( \log \frac{1+ \sqrt{z}}{1-\sqrt{z}}\Big)^2
             =1 + \frac{8}{\pi^2} z + \frac{16}{3 \pi^2} z^2 + \frac{184}{45 \pi^2} z^3+\dotsb.\]
This yields $B_1= {8}/{\pi^2} $ and $B_2 = {16}/{3 \pi^2}$ and thus we get
\begin{align*}
 |T_2(2)| & \leq \big(128 (72+12 \pi ^2+5 \pi ^4)\big)/(9 \pi ^8)  \approx  1.01547
 \intertext{and} |T_3(1)| &  \leq 1 + 3072/\pi^8 + 512/(3 \pi^6) + 1088/(9 \pi^4)  \approx  2.74232
 \end{align*}
 for $f \in \mathcal{S}_{P}$. For $f \in \mathcal{UCV}$, we get
 \[ |T_2(2)| \leq {16 (576+96 \pi ^2+85 \pi ^4 )}/{(81 \pi ^8)}  \approx   0.204083 .\]

\item[3.7]  Yunus \cite{yunus} \emph{et al.} studied the class $\mathcal{S}^*_{lim}:=\mathcal{S}^*(1+ \sqrt{2}z +z^2/2)$ associated with the limacon $(4 u^2 + 4v^2 -8u -5)^2 +  8 (4 u^2 + 4 v^2 -12u -3)=0$. The class $\mathcal{K}_{lim}:=\mathcal{K}(1+ \sqrt{2}z +z^2/2)$. Clearly, in this case $ B_1= \sqrt{2} $ and $ B_2 = 1/2 $ and therefore, we get $|T_2(2)| \leq 57/16 = 3.5625$ and $|T_3(1)| \leq 135/16 = 8.4375 $ for $f \in \mathcal{S}^*_{lim}$. For $f \in \mathcal{K}_{lim}$, $|T_2(2)| \leq 97/144 = 0.673611$ and $|T_3(1)| \leq 323/144 = 2.24306$
%

\item[3.8]  Wani \emph{et al.}\cite{swami}, studied the class of functions defined by $\mathcal{S}^*_{Ne} := \mathcal{S}^* (\varphi_{Ne}(z))$ and $\mathcal{K}_{Ne} := \mathcal{K} (\varphi_{Ne})$, where the function $\varphi_{Ne}(z) := 1 +  z  - z^3/3$ maps the unit disk into the interior of the 2-cusped kidney shaped nephroid. Clearly, here  $B_1= 1$ and $ B_2 = 0$, thereby yielding $|T_2(2)| \leq 5/4 = 1.25$ and $|T_3(1)| \leq 15/4 = 3.75$ for $f \in \mathcal{S}^*_{Ne}$. For $f \in \mathcal{K}_{Ne}$, $|T_2(2)| \leq 5/18 = 0.277778 $ and $|T_3(1)| \leq 14/9 = 1.55556$.

\end{description}

\section*{Acknowledgement} The authors are thankful to the referees for their useful comments.



\end{document}